\documentclass[12pt]{amsart}
\usepackage{geometry}                
\usepackage{graphicx}
\usepackage{amssymb}
\usepackage{epstopdf}
\usepackage{pst-node}
\usepackage{tikz-cd} 
\newtheorem{thm}{Theorem}[section]

\newtheorem{lemma}[thm]{Lemma}

\theoremstyle{definition}

\newtheorem{defn}[thm]{Definition}
\newtheorem{remark}{Remark}

\title{Local indicability in the presence of diagrammatic reducibility}
\author{Jens Harlander and Stephan Rosebrock}

\begin{document}

\begin{abstract} If a complex $X$ is a subcomplex of a diagrammatically reducible 2-complex $Y$ that has locally indicable fundamental group, then $X$ has locally indicable fundamental group. This is a consequence of the Corson-Trace characterization of diagrammatic reducibility. In this paper we use a Corson-Trace like characterization of diagrammatic reducibility away from a subcomplex to obtain a considerable stronger result. 
We apply this to the question of local indicability in the context of Whitehead's asphericity conjecture. We show that an injective labeled oriented tree (LOT) that is diagrammatically reducible of degree 2, and all its quotients are as well, is locally indicable. 
\end{abstract}

\maketitle

\noindent MCS classification: 57M07, 57M05, 20F05, 20F06, 20F65

\section{Introduction}

A group is {\em locally indicable} if every non-trivial finitely generated subgroup surjects onto $\mathbb Z$. Local indicability is not a topologically hereditary property, however in conjunction with diagrammatic reducibility it is: If a complex $X$ is a subcomplex of a diagrammatically reducible 2-complex $Y$ that has locally indicable fundamental group, then $X$ is diagrammatically reducible and has locally indicable fundamental group (for background on topological and combinatorial vocabulary we refer the reader to the next section). This is a consequence of the Corson-Trace characterization of diagrammatic reducibility. In this paper we use a Corson-Trace like characterization of diagrammatic reducibility away from a subcomplex (see \cite{HarRoseDDR}) to obtain a considerable stronger result. 

We apply this to investigate local indicability of LOT-groups. The acronym LOT stands for {\em labeled oriented tree}. LOT complexes are spines of ribbon 2-disc complements and include the class of spines of classical knot complements. Background on labeled oriented trees is provided in the last section of the paper. The unresolved asphericity question of LOT complexes and its relevance to Whitehead's asphericity question has a long history. See Berrick-Hillman \cite{BerrickHillman}, Bogley \cite{Bogley}, and Rosebrock \cite{Ro18}. It is known that local indicability of the LOT group implies asphericity of the LOT complex (see Howie \cite{Howie}). Asphericity of injective LOT complexes has been established in \cite{HarRose2017}, and a stronger asphericity condition has been shown in \cite{HarRose2020}. This paper is motivated by the question whether injective LOT groups are locally indicable. We show that a reduced injective LOT that satisfies the combinatorial asphericity condition DR(2), and all its quotients do as well, is locally indicable.

\section{Preliminaries}

A map $f\colon X\to Y$ between complexes is {\em cellular} if it maps cells to cells
and it is {\em combinatorial} if it maps the interior of a cell in $X$ homeomorphically to the interior of a cell in $Y$. A complex $X$ is {\em combinatorial} if the attaching maps of cells are combinatorial. All complexes in this paper are assumed to be connected and combinatorial, with fixed cell orientations, unless stated otherwise. 

If $v$ is a vertex of a 2-complex $X$ the {\em link at v}, $lk(v,X)$, is the boundary of a regular neighborhood of $v$. It is a graph whose edges come from the corners of 2-cells. For that reason we refer to the edges of $lk(v,X)$ as {\em corners at $v$}. An oriented edge $e$ of $X$ with initial vertex $v$ and terminal vertex $w$ contributes a vertex $e^+\in lk(v,X)$ and a vertex $e^-\in lk(w,X)$.
If $v$ is the only vertex of $X$ we sometimes write $lk(X)$ instead of $lk(v,X)$.

An edge path $e_1\dots e_k$ in a graph is {\em reducible} if $e_j=\bar e_{j-1}$, for some $j$, where $\bar e$ denotes the edge $e$ with opposite orientation. Otherwise the path is called {\em reduced}. A {\em spherical diagram} over a 2-complex $X$ is a combinatorial map $f\colon S\to X$, where $S$ is the 2-sphere with a combinatorial cell structure. The spherical diagram is {\em reduced} if for every vertex $v\in S$ the edge path $f\colon lk(v,S)\to X$ is reduced. If the spherical diagram is not reduced, then $S$ contains an edge $e$ so that the two 2-cells containing $e$ are mapped to the same 2-cell in $X$ with opposite orientation. Then the edge $e$ is called a {\em folding edge}. $X$ is {\em diagrammatically reducible} (DR) if every spherical diagram over $X$ is reducible.

A subcomplex $L\subseteq X$ is {\em full} if every cell in $X$ whose boundary is in $L$ is already in $L$. Let $X$ be a 2-complex with a full subcomplex $L$. Then $X$ is {\em DR away from $L$} if every spherical diagram $f\colon S\to X$ that contains an edge $e$ so that $f(e)\not\in L$, also contains a folding edge $e'$ so that $f(e')\not\in L$. One consequence of DR away from $L$ is the fact that the inclusion $L\to X$ induces an injection on fundamental groups (Theorem 2.2 in \cite{HarRoseDDR}). 

The following result is due to Corson-Trace \cite{CorsonTrace}.

\begin{thm}\label{thm:CT} Assume the 2-complex $X$ is DR. Let $\tilde X$ be its universal cover. Then every finite subcomplex of $\tilde X$ collapses into the 1-skeleton $\tilde X^{(1)}$.
\end{thm}

The following relative version of the Corson-Trace result was shown in \cite{HarRoseDDR}, Theorem 3.4.

\begin{thm}\label{thm:relCT} Let $X$ be a 2-complex, $p\colon \tilde X\to X$ its universal covering, and $L$ be a full subcomplex of $X$. Assume $X$ is DR away from $L$. Then every finite subcomplex of $\tilde X$ collapses into $p^{-1}(L)\cup \tilde X^{(1)}$, and $p^{-1}(L)$ is fixed under the collapsing process.
\end{thm}

\begin{defn}
    If $f\colon S\to X$ is a spherical diagram and $s$ is an edge of $S$ then we call $f(s)=e$, an edge in $X$, the {\em label} carried by $s$. We say $X$ is DR($k$) if $X$ is DR away from any set of $k-1$ or fewer edges. 
\end{defn}

DR(1) simply means DR and DR($k$) implies DR($k-1$). Note that DR($k$) implies that every spherical diagram that contains $j\le k$ edges with mutually distinct labels also contains $j$ folding edges with mutually distinct labels. 
This can be seen as follows: Suppose $X$ is DR($k$) and that $f\colon S\to X$ carries mutually distinct labels $e_1,\ldots e_j$ where $j\le k$. If $k=1$ then $X$ is DR and the diagram contains a folding edge. If $k>1$ then $X$ is DR($k-1$) and therefore contains folding edges with distinct labels $e'_1,\ldots, e'_{j-1}$ by induction. Since the diagram contains $j$ edges with distinct labels, it contains an edge with a label $e$ not equal to any of the $e_1',\ldots, e'_{j-1}$. Since $X$ is DR away from $e_1',\ldots, e'_{j-1}$ it contains a folding edge labeled $e'$ not equal to any of the $e_1',\ldots, e'_{j-1}$. Thus $e'_1,\ldots, e'_{j-1}, e'$ are $j$ distinct labels on folding edges.

It was shown in \cite{HarRoseDDR}, Theorem 2.6,  that if $X$ is the presentation complex for a 1-relator presentation $P=\langle x_1,\ldots ,x_n \ |\ r\rangle$, where $r$ is reduced, involves all generators, and is not a proper power, then $X$ is DR($n$). This implies that any set of $n-1$ generators generates a free subgroup of rank $n-1$, which is the content of the Freiheitssatz.\\

Let $X$ be a 2-complex.  If we assign numbers $\omega(c)\in \mathbb R$, called angles (or weights), to the corners $c$ of the 2-cells of $X$ we arrive at an {\em angled 2-complex}. Curvature in an angled 2-complex is defined in the following way. If $v$ is a vertex of $X$ then $\kappa(v,X)$, the curvature at $v$, is
$$\kappa(v,X)=2-\chi(lk(v,X))-\sum \omega(c_i),$$ where the sum is taken over all the corners at $v$. If $d$ is a 2-cell of $X$ then $\kappa(d,X)$, the curvature of $d$, is 
$$\kappa(d,X)=\sum \omega(c_j)-(|\partial d|-2),$$ where the sum is taken over all the corners in $d$ and $|\partial d|$ is the number of edges in the boundary of the 2-cell. The combinatorial Gauss-Bonnet Theorem states that
$$2\chi(X)=\sum_{v\in X} \kappa(v,X) + \sum_{d\in X} \kappa(d,X).$$
This was first proven by Ballmann and Buyalo \cite{BallBuy}, and later observed by McCammond and Wise \cite{McCammondWise}.
Suppose we have a combinatorial map $X\to Y$ between 2-complexes. If $Y$ is an angled 2-complex then the angles in the 2-cells of $Y$ can be pulled back to make $X$ into an angled 2-complex. We call this angle structure on $X$ the one {\em induced} by the combinatorial map. 

\begin{defn}(Weight test, Gersten \cite{Ger87})
Let $X$ be an angled 2-complex. Then $X$ {\em satisfies the weight test} if 
\begin{enumerate}
    \item the curvature of every 2-cell is $\le 0$;
    \item for every vertex $v$: If $c_1\cdots c_n$ is a reduced cycle in $lk(v,X)$, then\\ $2-\sum_{i=1}^n \omega(c_i)\le 0$. 
\end{enumerate}
\end{defn}

If an angled 2-complex satisfies the weight test then it is DR. This result is due to Gersten \cite{Ger87}. Gersten called the numbers $\omega(c)$ weights and not angles, because they do not have to be $\ge 0$. An earlier version of the weight test is due to Sieradski \cite{S83}. He considered the case where the angles take on only the values $0$ and $1$, and called his test the coloring test. Wise showed in \cite{Wise04} that if $X$ satisfies the coloring test (he called it ``the Sieradski weight test") then it has the non-positive immersions property and hence $\pi_1(X)$ is locally indicable. This is not true for the general weight test. For example, it can be shown that the standard 2-complex $X$ of the Higman presentation 
$$\langle a, b, c, d\ |\ aba^{-1}=b^2, bcb^{-1}=c^2, cdc^{-1}=d^2, dad^{-1}=a^2 \rangle$$
can be made to satisfies the weight test, but $\pi_1(X)$ is infinite and perfect, and $\chi(X)=1>0$. So $X$ can not be made to satisfy the coloring test.

An angled 2-complex where all angles are either 0 or 1 is called a {\em zero/one angled 2-complex}.
We denote by 
$lk_0(v,X)$ the subgraph of $lk(v,X)$ consisting of the vertices of $lk(v,X)$ together with the corners with angle $0$. 

\begin{thm}\label{thm:npi} A zero/one-angled 2-complex $X$ satisfies the coloring test if and only if 
\begin{enumerate}
\item the curvature of every 2-cell is $\le 0$;
\item $lk_0(v,X)$ is a forest for every vertex $v$;
\item a corner with angle $1$ does not have both its vertices in a single connected component of $lk_0(v,X)$.
\end{enumerate}
\end{thm}

The proof is straightforward. We end this section with two results that we will need in applications later on. Given an angled 2-complex $X$ and a path $\alpha\in lk(v,X)$ we define $\omega (\alpha)$ to be the sum of the weights of the corners that appear in $\alpha $. We define $\omega(lk(v,X))$ to be the sum of angles of all corners of $lk(v,X)$.

\begin{thm}\label{thm:DR(2)+} Let $X$ be an angled 2-complex with a single vertex $v$, where all angles are $\ge 0$. Assume that $X$ satisfies the weight test and that for every edge $e$, should there be a path $\alpha$ in $lk(v,X)$ that connects $e^+$ and $e^-$, its weight $\omega(\alpha)\ge 2$. Then $X$ is DR(2).
\end{thm}

\begin{proof} Let $f\colon S\to X$ be a spherical diagram. Since $\chi(S)=2>0$, the 2-sphere $S$ must contain a vertex $s$ of positive curvature by the combinatorial Gauss-Bonnet Theorem. Thus $\omega(lk(s,S))< 2$. It follows that $f(lk(s,S))$ is a subtree of $lk(v,X)$. This is because if $\beta$ were a reduced cycle in $f(lk(s,S))$, then $\omega(lk(s,S))\ge \omega(\beta)\ge 2$, which is false. It follows that $f(lk(s,S))$ contains at least two vertices $e_1^{\epsilon_1}$ and $e_2^{\epsilon_2}$ of valency 1, where the epsilons are in $\{ +, -\}$. Assume $e_1=e_2=e$. W.l.o.g. we assume that $e_1^{\epsilon_1}=e^+$. Then $e_2^{\epsilon_2}=e^-$ (because $e_1^{\epsilon_1}\ne e_2^{\epsilon_2}$), and we have a path $\alpha$ in $f(lk(s,S))$ connecting $e^+$ to $e^-$ such that $\omega(\alpha)\le \omega(f(lk(s,S)))<2$, which is ruled out. Thus $e_1\ne e_2$. We have shown that there are two folding edges in $S$ with distinct labels $e_1$ and $e_2$.
\end{proof}

\begin{thm}\label{thm:DR(2)zo} Let $X$ be a zero/one angled 2-complex with a single vertex $v$. Assume that $X$ satisfies the coloring test and that for every edge $e$ we have that $e^+$ and $e^-$ lie in different components of $lk_0(v,X)$. Then $X$ is DR(2) and $\pi_1(X)$ is locally indicable.
\end{thm}

\begin{proof} A 2-complex that satisfies the coloring test has the non-positive immersion property and therefore has locally indicable fundamental group. See Wise \cite{Wise04}, Theorem 11.4. We are left with showing DR(2). Let $f\colon S\to X$ be a spherical diagram. Since $\chi(S)=2>0$, the sphere $S$ must contain a vertex $s$ of positive curvature by the combinatorial Gauss-Bonnet Theorem. Thus $\omega(lk(s,S))< 2$. It follows from Theorem \ref{thm:npi} that $\omega(lk(s,S))=0$, which implies that the subgraph $f(lk(s,S))$ is contained in $lk_0(f(v,X))$, a forest. It follows that $f(lk(s,S))$ is a tree and therefore contains at least two vertices of valency 1, say $e_1^{\epsilon_1}$ and $e_2^{\epsilon_2}$, where the epsilons are in $\{ +, -\}$. Suppose $e_1=e_2=e$. We may assume w.l.o.g. that $e_1^{\epsilon_1}=e^+$. Then $e_2^{\epsilon_2}=e^-$, because $e_1^{\epsilon_1}\ne e_2^{\epsilon_2}$. But this implies that $e^+$ and $e^-$ lie in a single component of $lk_0(v,X)$, a contradiction. Thus $e_1\ne e_2$ and again there are two folding edges labeled differently in $S$.
\end{proof}

\begin{thm}\label{thm:DR(2)44}
    Let $X$ be a C(4)-T(4) 2-complex with a single vertex $v$. Suppose that the attaching maps for the 2-cells are reduced closed paths that do not contain an edge sequence $ee$, for any edge $e$. Then $X$ is DR(2) and $\pi_1(X)$ is locally indicable.
\end{thm}

\begin{proof}
    A C(4)-T(4) 2-complex has the non-positive immersions property and hence its fundamental group is locally indicable. See Wise \cite{Wise22}, Corollary 8.2. We have left to show DR(2). Note that a C(4)-T(4) 2-complex satisfies the weight test where all weights are $1/2$. Let $f\colon S\to X$ be a spherical diagram. Since $\chi(S)=2>0$, the sphere $S$ must contain a vertex $s$ of positive curvature by the combinatorial Gauss-Bonnet Theorem. Thus there is a vertex $s$ in $S$ of valency $\le 3$. Since $lk(v,X)$ does not contain cycles of length less than 4, it follows that $f(lk(s,S))$ is a tree, in fact a single edge in $lk(v,X)$ and the valence of $s$ is 2. Let $e_1^{\epsilon_1}$ and $e_2^{\epsilon_2}$ be the vertices of $f(lk(s,S))$. Assume that $e_1=e_2=e$. W.l.o.g. we assume that $e_1^{\epsilon_1}=e^+$. Then $e_2^{\epsilon_2}=e^-$. It follows that $X$ contains a 2-cell that is attached via a path that contains $ee$, which we do not allow. Thus $e_1\ne e_2$ and again there are two folding edges labeled differently in $S$.
\end{proof}


\section{Main result}

A group is called {\it locally free} if all its non-trivial finitely generated subgroups are free.

\begin{lemma}\label{lem:locfree}
Suppose $X$ is a 2-complex all of whose finite subcomplexes collapse into its 1-skeleton. Then $\pi_1(X)$ is locally free.
\end{lemma}

\begin{proof} 
This can be argued in more than one way, see Remark \ref{remark:limits} below. We present an elementary proof that, with slight variations, also applies to the next lemma.

Suppose $A$ is a finitely generated subgroup of $\pi_1(X)$ minimally generated by $n$ elements.
Let $B$ be a finitely presented group on $n$ generators that maps onto $A$. We can construct a finite 2-complex $W$ with fundamental group $B$, and a map $\alpha\colon W\to X$ so that the induced map 
$$\alpha_*\colon B=\pi_1(W)\to \pi_1(X)$$ has image $A$.
Let $Z=\alpha(W)$. Since $Z$ is a finite subcomplex of $X$, it collapses into the 1-skeleton $X^{(1)}$. 
So we may assume that $Z\subseteq X^{(1)}$. Our map $\alpha$ factors through $X^{(1)}$:
$$W\stackrel{\beta}{\to} X^{(1)}\stackrel{\iota}\hookrightarrow X$$
where $\iota$ is the inclusion and $\iota\circ \beta=\alpha$. Thus we have 
$$A=\alpha_*(\pi_1(W))=\alpha_*(B)=\iota_*\circ \beta_*(B)=\iota_*(\beta_*(B)).$$
Since $B$ is generated by $n$ elements, $\beta_*(B)$ is generated by $\le n$ elements. Since $\beta_*(B)$ maps onto $A$ it is generated by $\ge n$ elements, and therefore minimally generated by $n$ elements. Since $\beta_*(B)\subseteq \pi_1(X^{(1)})$, it is a free group of rank $n$, and therefore $B$ is free of rank $n$. We have shown that the only $n$-generator group $B$ that maps onto the $n$-generator group $A$ is free of rank $n$. It follows that $A$ is free of rank $n$.
\end{proof}

\begin{lemma}\label{lem:collapse} Let $L$ be a subcomplex of a 2-complex $X$. Assume that $\pi_1(L)$ is locally indicable and that every finite subcomplex of $X$ collapses into $L$. Then $\pi_1(X)$ is locally indicable.
\end{lemma}

\begin{proof}
Suppose $A$ is a finitely generated subgroup of $\pi_1(X)$ so that $H_1(A)$ is finite. Then there exists a finitely presented group $B$ that maps onto $A$ so that $H_1(B)$ is finite.
We can construct a finite 2-complex $W$ with fundamental group $B$ and a map $\alpha\colon W\to X$ so that the induced map 
$$\alpha_*\colon B=\pi_1(W)\to \pi_1(X)$$ has image $A$.
Let $Z=\alpha(W)$. Since $Z$ is a finite subcomplex of $X$, it collapses into $L$. 
So we may assume that $Z\subseteq L$. Our map $\alpha$ factors through $L$:
$$W\stackrel{\beta}{\to} L\stackrel{\iota}\hookrightarrow X$$
where $\iota$ is the inclusion and $\iota\circ \beta=\alpha$. Thus we have 
$$A=\alpha_*(\pi_1(W))=\alpha_*(B)=\iota_*\circ \beta_*(B)=\iota_*(\beta_*(B)).$$
Since $B$ is finitely generated, so is $\beta_*(B)$. Since $H_1(B)$ is finite, so is $H_1(\beta_*(B))$. 
Since $\beta_*(B)$ is a finitely generated subgroup of the locally indicable group $\pi_1(L)$ it follows that $\beta_*(B)=1$. Thus $A=\iota_*(\beta_*(B))=1$.
\end{proof}

\begin{remark}\label{remark:limits}
Lemma \ref{lem:locfree} also follow from direct limits. $X$ is a union of finite subcomplexes $X_i$ forming a directed set with respect to inclusion. We have
$$\pi_1(X)=\pi_1(\lim_{\to}X_i)=\lim_{\to}\pi_1(X_i).$$
Since each $X_i$ collapses into the 1-skeleton, each $\pi_1(X_i)$ is free and so $\pi_1(X)$ is a direct limit of free groups and therefore is locally free. In fact, a direct limit of locally free groups is locally free. See for example Conner et al. \cite{Conner}, Lemma 24.

It is not clear how to make the limit argument work for Lemma \ref{lem:collapse}, because despite the fact that every $X_i$ collapses into $L$ we can not conclude in general that $\pi_1(X_i)$ is locally indicable. Under the assumption that $L$ is DR in addition to $\pi_1(L)$ being locally indicable, the limit argument does work. If $L_i$ is the subcomplex of $L$ into which $X_i$ collapses, then $\pi_1(L_i)$ is locally indicable by Theorem \ref{thm:easycase} below. It follows that $\pi_1(X_i)$ is locally indicable and so $\pi_1(X)$ is a direct limit of locally indicable groups and therefore is locally indicable.
 \end{remark}

Here is a baby version of our main result Theorem \ref{thm:good} stated at the end of this section.

\begin{thm}\label{thm:easycase} Let $X$ be a subcomplex of the 2-complex $Y$. If $Y$ is DR then the kernel of the inclusion induced map $\pi_1(X)\to \pi_1(Y)$ is locally free. Consequently, if $\pi_1(Y)$ is locally indicable, then so is $\pi_1(X)$.
\end{thm}

\begin{proof} This is an immediate consequence of Theorem \ref{thm:CT}. Let $\bar X$ be the covering complex whose fundamental group is the kernel of the inclusion induced map $\pi_1(X)\to \pi_1(Y)$. Then, by standard covering space theory (see Hatcher \cite{hatcher}), we have $\bar X\subseteq \tilde Y$, where $\tilde Y$ is the universal covering of $Y$. Since every finite subcomplex of $\tilde Y$ collapses into the 1-skeleton by Theorem \ref{thm:CT}, every finite subcomplex of $\bar X$ has free fundamental group. It follows from Lemma \ref{lem:locfree} that $\pi_1(\bar X)$ is locally free and hence locally indicable. Thus the kernel of the inclusion induced map $\pi_1(X)\to \pi_1(Y)$ is locally indicable, and if we assume that $\pi_1(Y)$ is as well, we can conclude that $\pi_1(X)$ is locally indicable.
\end{proof}

\begin{defn} Let $f\colon X\to Y$ be a cellular map between 2-complexes. Suppose $K$ is a full subcomplex of $Y$ and $f^{-1}(K)=L$. Then $f$ is called an {\em immersion outside of $L$} if
\begin{enumerate}
\item The interior of a cell in $X$ not in $L$ is mapped homeomorphically to the interior of a cell in $Y$;
\item If $x$ is a point in $X$ not in $L$, then $f\colon lk(x,X)\to lk(f(x),Y)$ is an embedding.
\end{enumerate}
\end{defn}

\begin{lemma}\label{lem:imoutside} Suppose we have a map $f\colon X\to Y$ between 2-complexes, $K\subseteq Y$ and $L=f^{-1}(K)$. Suppose further that $f$ is an immersion outside of $L$. Assume that $Y$ is DR away from $K$. Let $p_X\colon \bar X\to X $ be the covering associated with the kernel of the induced map $f_*\colon \pi_1(X)\to \pi_1(Y)$. Then every finite subcomplex of $\bar X$ collapses into $p_X^{-1}(L)\cup \bar X^{(1)}$.
\end{lemma}

\begin{proof} Let $p_Y\colon \tilde Y\to Y$ be the universal covering of $Y$. We have a commutative diagram
\[\begin{tikzcd}
\bar X \arrow{r}{\bar f} \arrow[swap]{d}{p_X} & \tilde Y \arrow{d}{p_Y} \\
X \arrow{r}{f} & Y
\end{tikzcd}
\]
where $\bar f$ is the lift of $f$. Since $f$ is an immersion outside of $L$ we have that $\bar f$ is an immersion outside of $p_X^{-1}(L)$. Let $Z$ be a finite subcomplex of $\bar X$ that is not already contained in $p_X^{-1}(L)\cup \bar X^{(1)}$. Then $Z'=\bar f(Z)$ is a finite subcomplex of $\tilde Y$ not contained in $p_Y^{-1}(K)\cup \tilde Y^{(1)}$. Since $Y$ is DR away from $K$, $Z'$ contains a free edge $\tilde e$ not in $p_Y^{-1}(K)$ by Theorem \ref{thm:relCT}. Let $\tilde d$ be the 2-cell that can be collapsed  by pushing in $\tilde e$. 
Let $\bar e$ be an edge in $Z$ so that $\bar f(\bar e)=\tilde e$. Let $\bar d_1,\ldots, \bar d_k$ be the 2-cells in $Z$ that contain $\bar e$ in their boundary. Let $m_{\bar e}$ be the midpoint of $\bar e$ and $m_{\tilde e}$ be the midpoint of $\tilde e$. Since $\bar f$ is an immersion outside of $p_X^{-1}(L)$ the map $\bar f\colon lk(m_{\bar e}, \bar X)\to lk(m_{\tilde e}, \tilde Y)$ is an embedding. Therefore $f\colon lk(m_{\bar e}, Z)\to lk(m_{\tilde e}, Z')$ is an embedding as well. Since $lk(m_{\tilde e}, \tilde Z')$ is a single half circle (because $\tilde e$ is a free edge of $\tilde d$), $lk(m_{\bar e}, Z)$ can not contain more than one half circle. Thus $k=1$ and $\bar d_1$ is the only 2-cell of $Z$ that contains the edge $\bar e$ in its boundary, and it does so exactly once. It follows that $\bar e$ is a free edge in $Z$.
\end{proof}

\begin{lemma}\label{lem:locind} Suppose $X$ is a 2-complex, $L$ is a subcomplex of $X$, and $p_X\colon \bar X\to X$ is a covering. If each connected component of $L$ has locally indicable fundamental group, then so does $p^{-1}_X(L)\cup\bar X^{(1)}$.
\end{lemma}

\begin{proof} 
Each component $\bar L_c$ of $p_X^{-1}(L)$ is a covering of some component $L_c$ of $L$. Since ${p_X}_*\colon \pi_1(\bar L_c)\to \pi_1(L_c)$ is injective and we assumed that $\pi_1(L_c)$ is locally indicable, it follows that $\pi_1(\bar L_c)$ is locally indicable. Thus $\pi_1(p_X^{-1}(L)\cup\bar X^{(1)})$, being a free product of locally indicable groups and a free group, is locally indicable. 
\end{proof}

Here is our main result, a relative version of Theorem \ref{thm:easycase}.

\begin{thm}\label{thm:good} Suppose we have a map $f\colon X\to Y$ between 2-complexes, $K\subseteq Y$ and $L=f^{-1}(K)$. Suppose further that $f$ is an immersion outside of $L$. Assume that $Y$ is DR away from $K$. Then:
\begin{enumerate}
    \item if $L$ is DR then so is $X$;
    \item if every component of $L$ has locally indicable fundamental group and $\pi_1(Y)$ is locally indicable, then so is $\pi_1(X)$.
\end{enumerate}
\end{thm}

\begin{proof} (1) Let $p_X\colon \bar X\to X $ be the covering associated with the kernel of the induced map $f_*\colon \pi_1(X)\to \pi_1(Y)$. By Lemma \ref{lem:imoutside} every finite subcomplex of $\bar X$ collapses into $p^{-1}_X(L)\cup\bar X^{(1)}$. If $L$ is DR then $p_X^{-1}(L)$ is DR, because it is a covering space, and therefore $p^{-1}_X(L)\cup\bar X^{(1)}$ is DR. It follows that $\bar X$ is DR, and therefore $X$ is DR. 

\noindent (2) If every component of $L$ has locally indicable fundamental group, then by Lemma \ref{lem:locind} we have that $\pi_1(p^{-1}_X(L)\cup\bar X^{(1)})$ is locally indicable. Since every finite subcomplex of $\bar X$ collapses into $p^{-1}_X(L)\cup\bar X^{(1)}$, it follows from Lemma \ref{lem:collapse} that $\pi_1(\bar X)$ is locally indicable. We have 
$$\ker f_*=\pi_1(\bar X)\hookrightarrow \pi_1(X)\stackrel{f_*}{\rightarrow} \pi_1(Y).$$
Since we assume that $\pi_1(Y)$ is locally indicable, so is $\pi_1(X)$.
\end{proof}


\section{An application to labeled oriented trees}

A labeled oriented graph (LOG) $\Gamma = (E, V, s, t, \lambda)$ consists of two sets $E$, $V$ of edges and vertices, and three maps $s, t, \lambda\colon E\to V$ called, respectively source, target and label. $\Gamma$ is said to be a labeled oriented tree (LOT) when the underlying graph is a tree. The associated LOG presentation is defined as
$$P(\Gamma)=\langle V\ |\  s(e)\lambda(e)=\lambda(e)t(e),\ e\in E \rangle.$$
The LOG complex $K(\Gamma)$ is the standard 2-complex defined by the presentation, and the group $G(\Gamma)$ presented by $P(\Gamma )$ is equal to $\pi_1(K(\Gamma))$. 

It is known that LOT-complexes are spines of ribbon 2-disc complements. See Howie \cite{How83}. So the study of LOTs is an extension of classical knot theory. Asphericity, known for classical knots, is unresolved for LOTs. The asphericity question for LOTs is of central importance to Whitehead's asphericity conjecture: A subcomplex of an aspherical 2-complex is aspherical. See Berrick/Hillman \cite{BerrickHillman}, Bogley \cite{Bogley}, and Rosebrock \cite{Ro18}.

A {\em sub-LOG} $\Gamma_0=(E_0, V_0)\subseteq\Gamma$ is a subgraph so that $E_0\ne\emptyset$ and $\lambda\colon E_0\to V_0$. A {\em sub-LOT} is a sub-LOG which is a LOT itself. A LOG is called {\em boundary reduced} if whenever $v$ is a vertex of valency 1 then $v=\lambda(e)$ for some edge $e$. It is called {\em interior reduced} if for every vertex $v$ no two edges starting or terminating at $v$ carry the same label. It is called {\em compressed }if for every edge $e$ the label $\lambda(e)$ is not equal to $s(e)$ or $t(e)$. Finally, a LOG is {\em reduced} if it is boundary reduced, interior reduced, and compressed. Given a LOG, reductions can be performed to produce a reduced LOG, and, in case the LOG is a LOT, this process does not affect the homotopy type of the LOT complex. A LOG is called {\em injective} if the labeling map $\lambda\colon E\to V$ is injective. \\

\begin{thm}\label{thm:HuckRose} If $\Gamma$ is a reduced injective LOT that does not contain boundary reducible sub-LOTs, then $K(\Gamma)$ admits a zero/one-angle structure that satisfies the coloring test. It follows that $K(\Gamma)$ is DR(2) and $G(\Gamma)$ is locally indicable.
\end{thm}

\begin{proof}
The fact that $\Gamma$ admits a zero/one-angle structure that satisfies the coloring test is due to Huck-Rosebrock. See Theorem 3.2 and Theorem 3.3 of \cite{HarRose2022} and the references therein. Zero/one-angled 2-complexes that satisfy the coloring test have the non-positive immersions property and hence have locally indicable fundamental group.
See Wise \cite{Wise04}, Theorem 11.4.

The DR(2) statement has not appeared elsewhere and we will provide the details. We want to use Theorem \ref{thm:DR(2)zo}. We need to show that $x^+$ and $x^-$ lie in different components of $lk_0(K(\Gamma))$ for every edge $x\in K(\Gamma)$. Lets take a closer look at $lk(K(\Gamma))$.
It was shown in \cite{HarRose2022}, Theorem 3.3, that $K(\Gamma)$ has the following local bi-forest property: If $x_1,\ldots, x_n$ are the edges of $K(\Gamma)$, then there exists a choice of $\epsilon_i\in \{ +,-\}$ so that 
$\Lambda_1=\Lambda(x_1^{\epsilon_1},\ldots, x_n^{\epsilon_n})$ and $\Lambda_2=\Lambda(x_1^{-\epsilon_1},\ldots, x_n^{-\epsilon_n})$ are forests. Here $\Lambda(x_1^{\epsilon_1},\ldots, x_n^{\epsilon_n})$ is the subgraph of $lk(K(\Gamma))=\Lambda$ spanned by the vertices $x_1^{\epsilon_1},\ldots, x_n^{\epsilon_n}$. Furthermore, a zero/one-angle structure can be put on $K(\Gamma)$ (see \cite{HarRose2022}, Theorem 3.2) so that
$$lk_0(K(\Gamma))=\Lambda_1\cup\Lambda_2.$$ Since $\Lambda_1$ and $\Lambda_2$ are disjoint, it follows that $x^+$ and $x^-$, the $x$ being one of the $x_i$, lie in different components of $lk_0(K(\Gamma))$.
\end{proof}

\begin{thm}\label{thm:LOT44}
    Let $\Gamma$ be a reduced LOT so that $K(\Gamma)$ is C(4)-T(4). Then $K(\Gamma)$ is DR(2) and $G(\Gamma)$ is locally indicable
\end{thm}

\begin{proof}
    This follows directly from Theorem \ref{thm:DR(2)44}.
\end{proof}

Let $\Gamma$ be a reduced injective LOT and $\Gamma_1$ a sub-LOT, which may not be boundary reduced. Let $y$ be the vertex in $\Gamma_1$ that does not occurs as an edge label in $\Gamma_1$ (it might occur as an edge label in $\Gamma$). Let $\hat \Gamma$ be obtained from $\Gamma$ by collapsing $\Gamma_1$ to the single vertex $y$. The quotient map $\Gamma \to \hat \Gamma$ induces a map $f\colon K(\Gamma)\to K(\hat\Gamma)$ that is an immersion outside of $K(\Gamma_1)\subseteq K(\Gamma)$. Note that $f(K(\Gamma_1))=y$, a single edge.

\begin{thm}\label{thm:LOT} Let $\Gamma$ be a reduced injective LOT and $\Gamma_1$ a sub-LOT. Let the quotient $\hat\Gamma$ be obtained from $\Gamma$ by collapsing $\Gamma_1$ to the vertex $y\in \Gamma_1$. Assume that $K(\hat\Gamma)$ is DR(2).
\begin{enumerate}
    \item If $K(\Gamma_1)$ is DR then so is $K(\Gamma)$;
    \item If $G(\Gamma_1)$ is locally indicable then so is $G(\Gamma)$.
\end{enumerate}
\end{thm}

\begin{proof}
The map $f\colon K(\Gamma)\to K(\hat\Gamma)$ is an immersion outside of $K(\Gamma_1)$, and $f(K(\Gamma_1))=y$, a single edge. By assumption $K(\hat\Gamma)$ is DR(2), and hence DR away from $y$. Both statements (1) and (2) now follow from 
Theorem \ref{thm:good}.
\end{proof}

LOTs that satisfy the conditions of Theorems \ref{thm:HuckRose}, or \ref{thm:LOT44} qualify for $\hat \Gamma$ and $\Gamma_1$ in the above theorem. 

\begin{lemma}\label{lem:goodquotient}
    Let $\Gamma$ be a reduced injective LOT which contains proper sub-LOTs. Then $\Gamma$ contains a proper sub-LOT $\Gamma_0$
 so that the quotient $\hat \Gamma$ is injective, compressed, and interior reduced.
 \end{lemma}

\begin{proof}
Any quotient of an injective LOT is injective and injective LOTs are interior reduced. Let $\Gamma_0$ be a maximal proper sub-LOT of $\Gamma$ and let $x$ be the vertex of $\Gamma_0$ that is not an edge label in $\Gamma_0$. If the quotient $\hat \Gamma$ is not compressed, there has to be a boundary vertex $y$ of $\Gamma_0$ and an edge $e$ of $\Gamma$ not in $\Gamma_0$ with vertex $y$ labeled by $x$. But then $\Gamma_1=\Gamma_0\cup\{ e \}$ is a sub-LOT that contains $\Gamma_0$ and hence, by maximality of $\Gamma_0$, $\Gamma_1=\Gamma$. But $\Gamma_1$ is not boundary reduced, contradicting the assumption that $\Gamma$ is reduced. 
\end{proof}

\begin{lemma}\label{lem:union}
    Let $\Gamma$ be an injective reduced LOT that contains a sub-LOT $\Gamma_0$ so that the quotient $\hat \Gamma$ is not boundary reduced. Then $\Gamma$ is a union of proper sub-LOTs $\Gamma=\Gamma_1\cup \Gamma_2$ so that the intersection $\Gamma_1\cap \Gamma_2=\{ x \}$ is a single vertex. Thus $G(\Gamma)=G(\Gamma_1)*_{\mathbb Z} G(\Gamma_2)$ and, in case both $G(\Gamma_i)$, $i=1,2$, are locally indicable, then so is $G(\Gamma)$.
\end{lemma}

\begin{proof}
    Since $\hat \Gamma$ is not boundary reduced, there exists a boundary vertex $x$ that does not occur as an edge label in $\hat\Gamma$. The LOT $\Gamma$ is obtained from $\hat \Gamma$ by expanding a vertex $y$ to $\Gamma_0$. Suppose $y\ne x$. Then $x$ does not occur as an edge label in $\Gamma_0$ because $\Gamma_0$ is a sub-LOT of $\Gamma$. But then $x$ is a boundary vertex of $\Gamma$ that does not occur as an edge label, which contradicts the fact that $\Gamma$ is reduced. It follows that $y=x$, $\hat \Gamma$ is a sub-LOT of $\Gamma$, $\Gamma=\hat\Gamma\cup \Gamma_0$, and the intersection is $\{ x \}$. This implies that $G(\Gamma)=G(\hat \Gamma)*_{\mathbb Z} G(\Gamma_0))$. The statement about local indicability of $G(\Gamma)$ follows from Theorem 4.2 in Howie \cite{How82}.
\end{proof}

\begin{thm}\label{thm:DR(2)li}
    If LOT complexes of reduced injective LOTs are DR(2), then LOT groups of reduced injective LOTs are locally indicable. 
\end{thm}

\begin{proof}
    We will do induction on the number of vertices. If $\Gamma$ consists of a single vertex, then $G(\Gamma)=\mathbb Z$ which is locally indicable. Let $\Gamma$ be a reduced injective LOT. If $\Gamma$ does not contain a proper sub-LOT, then $G(\Gamma)$ is locally indicable by Theorem \ref{thm:HuckRose}.
    
    Assume $\Gamma$ contains a proper sub-LOT. Then by Lemma \ref{lem:goodquotient} $\Gamma$ contains a proper sub-LOT $\Gamma_0$ so that the quotient $\hat \Gamma$ is injective, compressed, and interior reduced. If the quotient $\hat \Gamma$ is not boundary reduced then $G(\Gamma)$ is locally indicable by Lemma \ref{lem:union} and induction. Next suppose that the quotient $\hat \Gamma$ is boundary reduced. Then $\hat \Gamma$ is reduced and injective.
    
    We check the conditions of Theorem \ref{thm:LOT}. By assumption $\hat \Gamma$ is $DR(2)$. If $\Gamma_0$ is not boundary reduced and  $\Gamma_0'$ is obtained from $\Gamma_0$ by doing boundary reductions, then $\Gamma_0'$ is reduced and injective and contains fewer vertices than $\Gamma$. By induction $G(\Gamma_0')$ is locally indicable. Since $G(\Gamma_0)=G(\Gamma_0')$ it follows that $G(\Gamma_0)$ is locally indicable. The conditions of Theorem \ref{thm:LOT} hold, and we conclude that $G(\Gamma)$ is locally indicable. 
\end{proof}

Reduced injective LOTs are known to be vertex aspherical VA, a combinatorial asphericity notion close to DR. See \cite{HarRose2020}. It is unknown if VA can be strengthened to DR or even DR(2).

\bigskip\noindent Jens Harlander, Boise State University,  1910 University Drive, Boise ID 83725-1555, USA.
Email: jensharlander@boisestate.edu

\bigskip\noindent Stephan Rosebrock, P\"adagogische Hochschule Karlsruhe, Bismarckstr. 10, 76133 Karlsruhe, Germany. Email: rosebrock@ph-karlsruhe.de

\end{document}